\documentclass{amsart}
\usepackage{amsmath, amsfonts, amssymb,amsthm}

\newtheorem{theorem}{Theorem}
\newtheorem{lemma}[theorem]{Lemma}
\newtheorem{cor}[theorem]{Corollary}

\theoremstyle{definition}
\newtheorem{exam}[theorem]{Example}

\newtheorem{remark}[theorem]{Remark}
\newtheorem{question}{Question}
\newtheorem{conjecture}[theorem]{Conjecture}

\def\notdiv{{\hbox{$\not|\,$}}}
\def\min{\mathop{\mathrm{min}}}
\def\CC{\mathbb C}

\def\ZZ{\mathbb Z}
\def\PP{\mathbb P}
\def\K{{\bf K}}
\def\x{{\bf x}}

\def\kk{{\bf k}}

 \def\ord{{\rm ord}}
 
\def\p{\mathbf p}
\def\q{\mathbf q}

\def\gen{\mathfrak g}

\def\Ocal{\mathcal O}

\def\Mcal{\mathcal M}
\def\Rcal{\mathcal R}
\let\a\alpha
\let\b\beta

\let\e\epsilon

\let\g\gamma

\let\l\lambda

\let\r\rho

\begin{document}
\title{ On the exponential local-global principle for meromorphic functions and algebraic functions}
\author{Hsiu-Lien Huang }
\address{Institute of Mathematics\\Academia Sinica\\6F, Astronomy-Mathematics Building\\No. 1, Sec. 4, Roosevelt Road\\ Taipei
10617\\Taiwan} \email{hlhuang@math.sinica.edu.tw}
\author{Andreas Schweizer }
\address{Department of Mathematics\\Korea Advanced Institute of Science and Technology (KAIST)\\ Daejeon 305-701
 \\South Korea} \email{schweizer@kaist.ac.kr}
\author{Julie Tzu-Yueh Wang}
\address{Institute of Mathematics\\Academia Sinica\\6F, Astronomy-Mathematics Building\\No. 1, Sec. 4, Roosevelt Road \\ Taipei
10617\\Taiwan} \email{jwang@math.sinica.edu.tw}
\thanks{The first and third author are partially supported by Taiwan's NSC grant 101-2115-M-001 -001 -MY2.
The second author is supported by ASARC in South Korea. }
\begin{abstract}
We prove the rank one case of Skolem's Conjecture on the exponential
local-global principle for algebraic functions and discuss its analog
for meromorphic functions.
\end{abstract}
\thanks{2010\ {\it Mathematics Subject Classification.} Primary 30D35
Secondary 11J97, 11D61, 11R58.}

\baselineskip=15truept
\maketitle \pagestyle{myheadings}
\markboth{HSIU-LIEN HUANG, ANDREAS SCHWEIZER and JULIE TZU-YUEH WANG }{ EXPONENTIAL LOCAL-GLOBAL PRINCIPLE }


\section{Introduction and results}
Let $K$ be a number field and $S$ a finite set of places of $K$ containing 
all the archimedean places. Denote by
  $\Ocal_S:=\{\a\in K: |\a|_v\le 1 \text{ for places } v\notin S\}$
the ring of $S$-integers and by
 ${\Ocal^*_S}:=\{\a\in K: |\a|_v= 1 \text{ for places } v\notin S\} $
the group of $S$-units.
 Let $\l_1,\hdots,\l_m$ be non-zero elements in $\Ocal_S$  and   $\a_1,\hdots,\a_m$ in  ${\Ocal^*_S}.$
 For every $n\in\ZZ$, we consider the following {\it power sum} with respect to $\l_i$ and $\a_i$,
 $$
 A(n):=\l_1\a_1^n+\cdots+\l_m\a_m^n\in\Ocal_S.
 $$
The following conjecture was suggested by Skolem   in \cite{Sk}.
\begin{conjecture}[Exponential local-global principle]\label{Skolem}
Assume that for every non-zero ideal $\frak a$ of the ring $\Ocal_S$, there exists $k \in \ZZ$ such that
$A(k)\in \frak a$.  Then there exists $n\in \ZZ$ such that $A(n)=0$.
\end{conjecture}
Recently, Bartolome, Bilu and Luca \cite{BBF} proved this conjecture for the case
when the rank of the multiplicative group generated by $\a_1,\hdots,\a_m$  is one.
Their proof relies on the gcd theorem of Corvaja and Zannier \cite{CZ} and a celebrated inequality of Baker (see the first two contributions in \cite{Wu}).
We refer to  \cite{BBF}  for the statements of these theorems and also a survey of results related to Conjecture \ref{Skolem}.

Due to some interesting analogy between Diophantine approximation and Nevanlinna theory,
we are also interested in the corresponding statements for complex meromorphic functions.
Denote by  $\Rcal$  the ring of entire functions and by $\Rcal^*$ its group of
units (i.e. entire functions without zeros).
Also, $\Mcal$ is the field of functions that are meromorphic in the complex
plane. Occasionally we write $U_n$ for the group of $n$-th roots of unity.
\par
Let $\l_1,\hdots,\l_m$ be non-zero elements in $\Rcal$ and $f_1,\hdots,f_m$ 
in  ${\Rcal^*}.$
 For every $n\in\ZZ$, we define the {\it power sum} with respect to $\l_i$ 
and $f_i$ by
 $$
 B(n):=\l_1f_1^n+\cdots+\l_m f_m^n\in\Rcal.
 $$
It is tempting to ask the following question which seems to be the analogue of Conjecture \ref{Skolem} in the complex case.
\begin{question}\label{SkolemC}
Assume that for every non-zero ideal $\frak a$ of the ring $\Rcal$, there exists $k \in \ZZ$ such that
$B(k)\in \frak a$.  Then there exists $n\in {\mathbb Z}$ such that $B(n)=0$.
\end{question}
However, this question might be too naive since it has an easy answer due to
the following reason:  The assumption implies that for each $z_0 \in\mathbb C$
there  exists $k \in \ZZ$ such that $B(k)(z_0)=0$. Therefore, there exists
$n\in {\mathbb Z}$ such that $B(n)$ has uncountably 
many zeros which is impossible as $B(n)$ is an entire function.  This observation brings out the difference between the complex case and the arithmetic case, and also indicates that the assumption should be loosened up.  Instead of considering all ideals in $\Rcal$, we will restrict ourselves to more specific conditions.
\par
Let more generally $f_i, \l_i \in\Mcal^*$.
If the multiplicative group generated by the $f_i$ has rank $1$,
it is a finitely generated group of rank $1$, so it is a direct
product of its (finite) torsion subgroup and an infinite cyclic
group. But a torsion element in $\Mcal^*$ can only be a function
that is identically equal to a root of unity. Choosing a generator
$f$ of the infinite cyclic part, we can write
$$f_i =\e_i f^{r_i}$$
with $\e_i$ a root of unity and $r_i\in\ZZ$.
This form will be the most convenient to formulate our results.
\par
Basic notation, definitions and results will be collected in
Section \ref{Nevanlinna}.

  \begin{theorem}\label{main}
Let $f$ be a non-constant meromorphic function and  $\l_1,\hdots,\l_m$
meromorphic functions such that there exists $\varrho<1$ with
$$T_{\l_i}(r)\leq\frac{\varrho}{m+\widetilde{m}}T_f(r)+S_f(r)$$
for $i=1,\ldots,m$,
where $\widetilde{m}$ is the number of $\l_i$ that are not entire.
Fix an m-tuple of integers $(r_1,\hdots,r_m)$ and an m-tuple of roots
of unity $(\e_1,...,\e_m)$ and let
$$
 B(n):=\l_1 (z) (\e_1f^{r_1 }(z))^n+\cdots+\l_m (z)(\e_mf^{r_m}(z))^n.
$$
Let $N=\max\{r_1,\hdots,r_m\}-\min\{r_1,\hdots,r_m\}$. Fix  an integer
$e>\frac{2}{1-\varrho}$ such that $\e_i^e=1$ for all $i=1,2,\ldots,m$.
Let $a$ be a positive integer such that
$$p^{1+\ord_p a -\ord_p e} > N+\frac{2}{1-\varrho}$$
for every prime divisor $p$ of $e$. (This implies of course $e|a$.)
Assume that there exists an integer $k$ such that $B(k)$ vanishes at every
zero of $f^a -1$.
\par
If $e$ divides $k$, then $\l_1 + \cdots +\l_m \equiv 0$, i.e. $B(0)$
vanishes identically.
\par
If $e$ does not divide $k$, then $\sum\l_i \e_i^k X^{r_i}\in\Mcal [X,X^{-1}]$
is the zero polynomial; so in particular $B(k)$ vanishes identically.
\end{theorem}

\begin{remark}
 \begin{enumerate}
\item If (at least) one $r_i$ is different from all the other ones and the
corresponding $\l_i$ does not vanish identically, then the case $e\notdiv k$,
that is $\sum\l_i \e_i^k X^{r_i}$ being the zero polynomial, can of course not
occur.
\item When $f$ is an entire function, we can use $\infty$ as one of the
points when applying the truncated second main theorem in the proof of
Theorem~\ref{main} and obtain $q-1\leq q\varrho$, that is
$q\leq\frac{1}{1-\varrho}$. So in this case the conditions can be relaxed
to $e>\frac{1}{1-\varrho}$ and $p^{1+\ord_p a -\ord_p e} > N+\frac{1}{1-\varrho}$.
\par
If $f$ is an entire function without zeroes, they can even be relaxed to
$e>0$ and $p^{1+\ord_p a -\ord_p e} > N$. If in that case moreover all $\e_i$
are equal to $1$, we can take $e=1$, and then there are no conditions at
all on $a$.
In this special case, vanishing of one $B(k)$ at the zeroes of $f-1$ does
indeed imply $\l_1 + \cdots +\l_m \equiv 0$.
\end{enumerate}
\end{remark}

\begin{exam}
Let
$$B(n)=f^{4n}+f^{3n}+(-f^2)^n+(-f)^n = f^n (f^n +1)(f^{2n}+(-1)^n).$$
Obviously no $B(k)$ vanishes identically. But if $k$ is odd, then
$B(k)$ vanishes at all zeroes of $f^{2k}-1$. So, although $\e_i^2=1$,
we cannot avoid requiring $4|a$ in Theorem~\ref{main}.
\par
Also, if $a$ is any positive integer, then $B(\frac{a}{\gcd(2,a)})$
vanishes at all zeroes of $f-\xi$ for all {\it primitive} $a$-th roots
of unity $\xi$.
\end{exam}

In Theorem~\ref{main} we assume that for a certain integer $a$ there
exists an integer $k$ such that $B(k)$ vanishes at all zeroes of
$f^a -1$. Then we prove that {\it this} $B(k)$ or $B(0)$ vanishes
identically. The following example shows that the growth condition
in Theorem~\ref{main} is the weakest one under which such a statement
holds.

\begin{exam}
Let $f=e^{2z}$ and
$$B(n) = e^{z}f^{2n}-e^{-z}f^n = e^{-z}f^n(f^{n+1}-1).$$
Then $T_{\l_i}(r) \leq \frac{1}{m}T_f(r)+S_f(r)$.
For every positive integer $a$ there exist infinitely many nonnegative
integers $k$ such that $B(k)$ vanishes at all zeroes of $f^a -1$.
(Simply take $k=a-1,\ 2a-1,\ \ldots$.)
But the only $B(n)$ that vanishes identically is $B(-1)$.
\end{exam}

   It is natural to consider next  the case when the coefficients $\l_i$ of $B(n)$ are of the same growth as $f$.   Unfortunately,  this seems to be out of reach, at least for now.  However, for the case of algebraic functions $f$ there is no difficulty to work out
the situation when the coefficients  $\l_i$ of $B(n)$  are algebraic functions.
The second part of this paper is to prove this function field analogue.

Let $\K$ be an algebraic function field (of one variable) over an algebraically closed field $\kk$. Let $C$ be the smooth projective curve defined over $\kk$ associated to $\K$ and write $\gen$ for the genus of $\K$ (or equivalently, of $C$).
For each point $\p\in C$, we may choose a uniformizer $t_\p$ to define a normalized order function $v_\p:=\ord_\p:\K\to \ZZ\cup\{\infty\}$ at $\p$.
Let $S$ be a set of finitely many points of $C$.
Denote by
  $\Ocal_S:=\{f\in \K: v_\p(f)\ge 0 \text{ for   } \p\in C\setminus   S\}$     the ring of $S$-integers and
 ${\Ocal^*_S}:=\{f\in \K:v_\p(f)= 0 \text{ for   } \p\in C\setminus S\} $     the  group of $S$-units.
Then Conjecture \ref{Skolem} can be formulated for the algebraic function field case identically.  In particular, we will prove
the following theorem which  is a stronger analogue of \cite{BBF} in function fields.

 \begin{theorem}\label{functionfields}
Let $\K$ be an algebraic function field with algebraically closed 
constant field $\kk$ of characteristic $0$.
Let $\l_1,\hdots,\l_m$ and $f$ be non-zero elements in $\K$.
Let $S$ be a set of finitely many points of $C$ such
that $\l_1,\hdots,\l_m\in \Ocal_S$ and $f$ is a unit in $\Ocal_S$.
Fix an m-tuple of integers $(r_1,\hdots,r_m)$ and an m-tuple
of roots of unity $(\e_1,...,\e_n)$ and let
$$
B(n):=\l_1(\e_1f^{r_1 })^n+\cdots+\l_m (\e_mf^{r_m})^n.
$$
Assume that for every positive integer $a$ there exists $k \in \ZZ$ such that
$v_\p(B(k))\ge \min\{1, v_\p(f^a -1)\}$ for all $\p\in C\setminus S$.
In other words, $B(k)$ vanishes at all zeros of $f^a -1$ that are not in $S$. 
Then there exists $n\in {\mathbb Z}$ such that $B(n)=0$.
 \end{theorem}

 \begin{remark}
 \begin{enumerate}
 \item
For the rank one case for number fields, it follows from the proof of  \cite{BBF} that  the local condition can be relaxed to
``for  every positive integer $a$ there exists $k \in \ZZ$ such that $B(k)$ is contained in the principal ideal $(\a^a -1)$  of $\Ocal_S$ where $\a\in\Ocal_S^*$ is the
generator of the free part of the rank one subgroup under consideration in \cite{BBF}.  Moreover, if one assumes the generalized $abc$-conjecture from \cite{Voj},
 one can conclude the same statement as the above theorem by adapting our arguments.
 \item
The analogue of \cite{BBF} for global function fields, i.e. $\kk$ is a finite field, was proved by Chia-Liang Sun recently in \cite{Sun} with a different approach.
 \end{enumerate}
 \end{remark}

When the height of $\l_i$ is ``small" compared to the height of $f$, one can
give a stronger result that, moreover, is valid in positive characteristic
as well.

 \begin{theorem}\label{functionfieldsconstant}
Let $\K$ be an algebraic  function field with algebraically closed constant 
field $\kk$ of characteristic $p\ge 0$.
Let $f$ be nonconstant in $\K$ and $S$ be a set of finitely many points of 
$C$ such that  $f\in\Ocal_S$.
Let $\l_1,\hdots,\l_m$  be non-zero elements in $\Ocal_S$ such that there exists $\r<1$ such that
$$ \sum_{i=1}^m h(\l_i)\leq \frac{\r h(f)}{\deg_{ins}(f)},$$
where $\deg_{ins}(f)=1$ if $p=0$ and $\deg_{ins}(f)=p^{\ell}$ if $p>0$ and $f\in \K^{p^{\ell}}\setminus \K^{p^{\ell+1}}.$
 Fix an m-tuple of integers $(r_1,\hdots,r_m)$ and an m-tuple
of roots of unity $(\e_1,...,\e_n)$ and let
$$
B(n):=\l_1(\e_1f^{r_1 })^n+\cdots+\l_m (\e_mf^{r_m})^n.
$$
Let $N=\max\{r_1,\hdots,r_m\}-\min\{r_1,\hdots,r_m\}$. Fix  an integer
$e> \frac{2-\r}{1-\r}(2\gen+|S|)$ such that $\e_i^e=1$ for all $i=1,2,\ldots,m$.
(We may also assume that $e$ is not divisible by $p$ if $p>0$.)
Let $a$ be a positive integer such that
$$q^{1+\ord_p a -\ord_p e} > N+\frac{2-\r}{1-\r}(2\gen+|S|)$$
for every prime divisor $q$ of $e$. (This implies of course $e|a$.)
Assume that there exists an integer $k$ such that
$v_\p(B(k))\ge \min\{1, v_\p(f^a -1)\}$ for all $\p\in C\setminus S$.
\par
If $e$ divides $k$, then $\l_1 + \cdots +\l_m \equiv 0$, i.e. $B(0)$
vanishes identically.
\par
If $e$ does not divide $k$, then $\sum\l_i \e_i^k X^{r_i}\in\K [X,X^{-1}]$
is the zero polynomial; so in particular $B(k)$ vanishes identically.
 \end{theorem}

The structure of the proof of Theorem~\ref{functionfields} is based on \cite{BBF}.
The first ingredient, the gcd theorem in  function fields of characteristic zero, was established  again by Corvaja and Zannier in \cite{CZ2}.
The second ingredient of the proof in \cite{BBF}, i.e. the  inequality of Baker,   will be replaced by the (truncated) second main theorem for algebraic functions from \cite{Wa}.
 Indeed,  the  inequality of Baker  used in  \cite{BBF}  can be replaced by Roth's theorem.
The proof of Theorem~\ref{main} relies on Nevanlinna's theory.

\section{The Complex Case}
\def\theequation{2.\arabic{equation}}
\setcounter{equation}{0}
\subsection{Preliminaries on Nevanlinna Theory}\label{Nevanlinna}
\def\theequation{2.\arabic{equation}}
\setcounter{equation}{0}
We will set up some notation and definitions in
 Nevanlinna theory for complex meromorphic functions and recall some basic results.
We refer to   \cite[Chapter VI ]{La} or \cite[Chapter 1]{Ru} for details.

Let   $f$ be a meromorphic function  and   $z\in \CC$. Denote by
$$\ord_z^+ (f):=\max\{0,\ord_z(f)\}, \quad\text{and }\quad
  \ord_z^- (f):=-\min\{0,\ord_z(f)\}.
$$
The {\it counting function} of $f$ at $\infty $ is defined by
$$N_f(r,\infty):=\sum_{0<|z|\le r } \ord_z^- (f)\log |\frac{r}{z}|+\ord_0^- (f)\log r.$$
Then define the {\it counting function} $N_f(r,a)$ for $a\in\CC$ to be
$$
N_f(r,a):=N_{1/(f-a)}(r, \infty).
$$
The {\it truncated counting function} of $f$ at $\infty $ is defined by
$$\overline{N}_f(r,\infty):=\sum_{0<|z|\le r } \min\{1,\ord_z^- (f)\}\log |\frac{r}{z}|+\min\{1,\ord_0^- (f)\}\log r.$$
Then define the {\it truncated counting function} $\overline{N}_f(r,a)$ for $a\in\CC$ to be
$$
\overline{N}_f(r,a):=\overline{N}_{1/(f-a)}(r, \infty).
$$
The  {\it proximity function} $m_f(r,\infty)$ is defined by
$$
m_f(r,\infty):=\int_0^{2\pi}\log^+|f(re^{i\theta})|\frac{d\theta}{2\pi},
$$
where $\log^+x=\max\{0,\log x\}.$ For any $a\in \CC,$ the {\it proximity function} $m_f(r,a)$ is defined by
$$m_f(r,a):=
m_{1/(f-a)}(r,\infty).
$$
The {\it characteristic function} of $f$ is defined by
$$
T_f(r):=m_f(r,\infty)+N_f(r,\infty).
$$
Finally, we write $S_f(r)$ for any real function $h(r)$ for which
$$\frac{h(r)}{T_f(r)}\to 0\ \ \hbox{\rm as}\ \  r\to\infty,\ r\notin E$$
where $E$ is an exceptional set of finite Lebesgue measure.
\par
We now recall the main theorems.
\begin{theorem}[{\bf First Main Theorem}]\label{Cfirstmain}  Let $f$ be a non-constant  meromorphic function on $\CC$.  Then for every $a\in\CC$, and any positive real number $r$,
$$
m_f(r,a)+N_f(r,a)=T_f(r)+O(1),
$$
where $O(1)$ is a constant independent of $r$.
\end{theorem}
\begin{theorem}[{\bf Truncated Second Main Theorem}]\label{SMT}  Let $f$ be a non-constant  meromorphic function on $\CC$, and $a_1,\hdots,a_q$ be distinct elements in $\CC\cup\{\infty\}$.
 Then for $r>0$,
$$
(q-2)T_f(r)\le_{exc}\sum_{i=1}^q \overline{N}_f(r,a_i)+O(\log^+ T_f(r)),
$$
where $ \le_{exc}  $ means the estimate holds except for $r$ in  a set of finite Lebesgue measure.
\end{theorem}

\subsection{Proof of Theorem \ref{main}}
\begin{proof}[Proof of Theorem \ref{main}]

Without loss of generality we assume
$r_1\leq r_2\leq \hdots\leq r_m$. Let
$$P_k(X):=\sum_{i=1}^m \l_i(z)\e_i^kX^{r_i -r_1}\in\Mcal[X].$$
Then $B(k)=f^{r_1 k}P_k(f^k)$.
By basic properties of characteristic functions, we have
$$T_{P_k(\gamma)}(r)\leq m\cdot \max\{T_{\l_1}(r),\cdots,T_{\l_m}(r)\}\leq
\frac{m\varrho}{m+\widetilde{m}}T_f(r)+S_f(r)$$
for any complex number $\gamma$.
\par
Now let $\xi$ be an $a$-th root of unity and $z_0$ a zero of $f-\xi$.
Then for each integer $\ell$,  $f^{\ell}(z)$ can be expressed as
$\xi^{\ell}+(z-z_0)g_{\ell}(z)$
where $g_{\ell}(z)$ is a meromorphic function with
$\ord_{z_0}g_{\ell}(z)\geq 0$.
From
$$\l_i(z)\e_i^k f^{r_i k}(z) =
\xi^{r_1 k}\l_i(z)\e_i^k (\xi^k)^{r_i -r_1}+
\l_i(z)\e_i^k (z-z_0)g_{r_i k}(z)$$
we obtain
$$B(k) = \xi^{r_1 k}P_k(\xi^k)+
\sum_{i=1}^m \e_i^k\l_i(z)(z-z_0)g_{r_i k}(z).$$
As $B(k)$ vanishes at $z_0$, we see that every zero of $f-\xi$ that
is not a pole of any $\l_i$ is a zero of $P_k(\xi^k)$. Consequently
$$\overline{N}_f(r,\xi)-\sum_{i=1}^m\overline{N}_{\l_i}(r,\infty)\leq
\overline{N}_{P_k(\xi^k)}(r,0).$$
Thus, either $P_k(\xi^k)\equiv 0$ or from Theorem~\ref{Cfirstmain}
and the above we get
\begin{eqnarray*}
\overline{N}_f(r,\xi) & \leq &
\overline{N}_{P_k(\xi^k)}(r,0)+\sum_{i=1}^m \overline{N}_{\l_i}(r,\infty)\\
 & \leq & T_{P_k(\xi^k)}(r)+\frac{\widetilde m\varrho}{m+\widetilde{m}}T_f(r)+S_f(r)
\leq \varrho T_f(r)+S_f(r).
\end{eqnarray*}
Let $\xi_1,\ldots,\xi_q$ be all $a$-th roots of unity with
$P_k(\xi^k)\not\equiv 0$. Then from Theorem~\ref{SMT} and the
above we get
$$(q-2)T_f(r)\leq\sum_{i=1}^q\overline{N}_f(r,\xi_i) +S_f(r)\leq
q\varrho T_f(r)+S_f(r).$$
This shows $q-2\leq q\varrho$, or equivalently
$q\leq\frac{2}{1-\varrho}$.
\par
If $e$ divides $k$, since $e>\frac{2}{1-\varrho}$ there exists at
least one $e$-th root of unity $\xi$ with
$$0\equiv P_k(\xi^k)\equiv P_k(1)\equiv \l_1(z)+\cdots+ \l_m(z),$$
which proves the first claim of the theorem.
\par
Now we treat the case $e\notdiv k$. Let $b=a/\gcd(a,k)$.
Since there exists a prime $p$ with $\ord_p e >\ord_p k$, by the
assumptions of the theorem, we have $b>N+\frac{2}{1-\varrho}$.
Now raising to the $k$-th power is a surjective homomorphism
from $U_a$ to $U_b$.
So there are at least $N+1$ different $b$-th roots of unity,
say $\beta_0, \beta_1,\ldots, \beta_N$, with
$$P_k(\beta_i)\equiv 0.$$
Since $P_k(X)\in\Mcal[X]$ is a polynomial of degree at most $N$,
we can write it as $P_k(X)=\sum_{j=0}^N c_j(z)X^j$ with $c_j(z)\in\Mcal$.
Then we have
$$\sum_{j=0}^Nc_j(z)\beta_i^j \equiv 0\ \ \ \hbox{\rm for}\ i=0,1,\ldots ,N.$$
This means that the vector in $\Mcal^{N+1}$ with components
$c_0(z),\ldots, c_N(z)$ lies in the kernel of the Vandermonde matrix
$M=((\beta_i^j))_{0\leq i,j\leq N}$. As $M$ is invertible, this is only
possible if all $c_j(z)$ vanish identically, i.e. if $P_k(X)$ is the
zero polynomial. So we have proved that $e\notdiv k$ implies
$B(k)\equiv 0$.
\end{proof}

\section{Algebraic function fields}
\def\theequation{3.\arabic{equation}}
\setcounter{equation}{0}
\subsection{Preliminaries on Value Distribution Theory}
\def\theequation{3.\arabic{equation}}
\setcounter{equation}{0}

Let $\K$ be an algebraic function field over an algebraically closed field $\kk$. Let $C$ be the smooth projective curve defined over $\kk$ associated to $\K$ and write $\gen$ for the genus of $\K$ (or equivalently, of $C$).

We first define valuations and height functions on the function field $\K$. For each point $\p\in C$, we may choose a uniformizer $t_\p$ to define a normalized order function $v_\p:=\ord_\p:\K\to \ZZ\cup\{\infty\}$ at $\p$.
For $f\in \K$, the (relative) height is defined by
$$
h(f)=h_\K(f):=\sum_{\p\in C}-\min\{0,v_\p(f)\}.
$$
Viewing $f$ as a morphism from $C$ to $\PP^1$, then  the degree of $f$ equals $h(f)$.
As the number of zeros (counting multiplicity) of an algebraic function $f$ equals its number of poles, we have
$$
h(f^{-1})=h(f).
$$
For $\x=[x_0:...:x_n]\in\PP^n(\K)$, the projective height is defined by
$$
h(\x)=h_\K(\x):=\sum_{\p\in C}-\min\{ v_\p(x_0),...,v_\p(x_n)\}
$$
which is independent of the choice of the representative vector $(x_0,...,x_n)$.

Let $A(X)=a_nX^n+a_{n-1}X^{n-1}+\cdots+a_1X+a_0\in \K[X]$  with $a_n \ne 0$. Denote by
\begin{align*}
v_\p(A):=  \min\{v_\p(a_0),...,v_\p(a_n)\}\qquad\text{for }\p\in C.
\end{align*}
The  height of   $A$ is defined by
 \begin{align*}
h(A):= h_K([a_0:...:a_n])=-\sum_{\p\in C} v_\p(A).
\end{align*}
We recall that Gauss' lemma in this context says that
 \begin{align}\label{Gauss}
v_\p(AB)=v_\p(A)+v_\p(B),
\end{align}
where $A$ and $B$ are in $\K[X]$ and $\p\in C$. Consequently, we have that
 \begin{align}\label{Gaussht}
h(AB)=h(A)+h(B).
\end{align}
In particular, if $A(X)=\prod_{i=1}^m(X-\b_i)$  with $\b_i\in\K$,
then
\begin{align}\label{Gausshtlinear}
h(A)=\sum_{i=0}^m h(\b_i).
\end{align}

Let $S$ be a set containing a finite number of  points of $C$  and  $\b$ a non-zero element in $\K$.  The  truncated counting function with respect to $ S$   is defined by
$$
  \overline{N}_{S}(\b):=\sum_{\p\in C\setminus  S}\min\{1,\max\{0,v_\p(\b)\}\}.
$$
The following  version of the second main theorem for $\K$ can be easily obtained from \cite[Theorem 1]{Wa}.
\begin{theorem}\label{SMTcff1}
Let  $p\ge 0$ be  the characteristic of $\K$,
 $S$ a set containing a finite number of  points of $C$ and $b_1,...,b_q$  
distinct elements in ${\bf k}$.
If $f\notin {\bf k} $, then
\begin{align*}
(q-2)\frac{h(f)}{\deg_{ins}(f)} 
\le \sum_{i=1}^q  \overline{N}_S(f-b_i)+\chi_S,
\end{align*}
where $\deg_{ins}(f)=1$ if $p=0$ and $\deg_{ins}(f)=p^{\ell}$ if $p>0$ and $f\in \K^{p^{\ell}}\setminus \K^{p^{\ell+1}}$, and  $\chi_S=2\gen-2+|S|$.
\end{theorem}

\begin{cor}\label{SMTcor}
Let  $p\ge 0$ be  the characteristic of $\K$,
 $S$ a set containing a finite number of  points of $C$ and $b_1,...,b_q$ 
distinct elements in ${\bf k}$.
 If $f\notin\kk$ and $f-b_i$, $1\le i\le q$, are $S$-units, then
 $q\le2\gen+|S|$.
 \end{cor}

Let $f$, $g$ be non-zero element in $\Ocal_S$. The counting function and the truncated counting function of the greatest common divisor of $f$ and $g$ in $\Ocal_S$
are  denoted by
$$
  N_{S}(\gcd(f,g)):=\sum_{\p\in C\setminus  S}\min\{v_\p(f),v_\p(g)\},
$$
and
$$
 \overline{N}_{S}(\gcd(f,g)):=\sum_{\p\in C\setminus  S}\min\{1,v_\p(f),v_\p(g)\}.
$$
The following theorem on greatest common divisors over function fields is due to Corvaja and Zannier \cite[Corollary 2.3]{CZ2}.
\begin{theorem}\label{CZgcd}
Let   the characteristic of $\K$ be zero and
 $S$ be a set containing a finite number of  points of $C$.
Let $a,b\in \Ocal_S^*$ be $S$-units, not both constant.  If $a,b$ are multiplicatively independent, we have
$$
N_{S}(\gcd(1-a,1-b))\le 3\sqrt[3]{2}(h(a)h(b)\chi_S)^{\frac13},
$$
where $\chi_S=2\gen-2+|S|$.
\end{theorem}

\subsection{Proofs of Theorems \ref{functionfields} and \ref{functionfieldsconstant}}

\begin{proof}[Theorem \ref{functionfieldsconstant}]
Without loss of generality we assume
$r_1\leq r_2\leq \hdots\leq r_m$. Let
$$P_k(X):=\sum_{i=1}^m \l_i \e_i^kX^{r_i -r_1}\in\K[X].$$
Then $B(k)=f^{r_1 k}P_k(f^k)$.
By basic properties of height functions, we have
$$
h(P_k(\g))\le \sum_{i=1}^m h(\l_i)\le  \frac{\r h(f)}{\deg_{ins}(f)}
$$
for any $\g\in\kk$.

By Corollary \ref{SMTcor}, there are at most $2\gen+|S|$  non-zero $c$ in $\kk$
such that $f-c$ is an $S$-unit as $f$ is non-constant in $\Ocal_S$.
Therefore, there are at least $a-2\gen-|S|$ $a$-th roots of unity $\xi$ such that
$f-\xi$ is not an $S$-unit.  Consequently,  $f-\xi$ has a zero outside of $S$.
Without loss of generality, we assume that $\xi_i$, $1\le i\le  a-2\gen-|S|$, are distinct $a$-th roots of unity such that
$f-\xi_i \in\Ocal_S\setminus\Ocal_S^*$.
\par
Now let $\p_0\in C\setminus S$ be 
a zero of $f-\xi_i $ for some $1\le i\le  a-2\gen-|S|$.
Then for each integer $j$,  $f^{j}$ can be expressed as
$\xi_i^{j }+g_{j}$
where $g_{j}\in\Ocal_S$  with
$v_{\p_0}(g_{j})\geq 1$.
From
$$\l_i \e_i^k f^{r_i k}  =
\xi^{r_1 k}\l_i \e_i^k (\xi^k)^{r_i -r_1}+
\l_i\e_i^k  g_{r_i k} $$
we obtain
$$B(k) = \xi^{r_1 k}P_k(\xi^k)+
\sum_{i=1}^m \e_i^k\l_ig_{r_i k}.$$
As $B(k)$ vanishes at $\p_0$, we see that $P_k(\xi_i^k)$ vanishes at $\p_0$ as well.
Consequently
$$
\overline{N}_S (f-\xi_i)\le \overline{N}_S(P_k(\xi_i^k)).
$$
Thus, either $P_k(\xi_i^k)\equiv 0$ or from the above
we get
\begin{align}
\overline{N}_S (f-\xi_i)&\le \overline{N}_S(P_k(\xi_i^k))\cr
&\le h(P_k(\xi_i^k))\le \sum_{i=1}^m h(\l_i)\le  \frac{\r h(f)}{\deg_{ins}(f)}.
\end{align}
Let $\xi_1$,...,$\xi_q$ be all $a$-th roots of unity with $f-\xi_i\not\in\Ocal_S^*$ and
$P_k(\xi_i^k)\not\equiv 0$.
Then from Theorem \ref{SMTcff1} and the above we get
$$
(q-2) \frac{  h(f)}{\deg_{ins}(f)}\le \sum_{i=1}^q \overline{N}_S(f-\xi_i)+2\gen-2+|S|\le q\r \frac{ h(f)}{\deg_{ins}(f)}+2\gen-2+|S|.
$$
This shows $q-2-q\r\le 2\gen-2+|S|$ since $ \frac{  h(f)}{\deg_{ins}(f)}\ge 1$.  Consequently,
$q\le \frac{2\gen+|S|}{1-\r}$.
\par
We may rearrange the order of the $a$-th roots of unity  $\xi_i$ again and assume that
$P_k(\xi_i^k)\equiv 0$ for $1\le i\le a- \frac{2-\r}{1-\r}(2\gen+|S|)$.
\par
The rest of the argument is identical to the last part of
the proof of Theorem \ref{main} and will be omitted.
\end{proof}

Some parts of the proof of Theorem \ref{functionfields} are similar to the number field case in \cite{BBF}.  We will omit some arguments that are similar to the number field case.
We first note that it suffices to prove Conjecture \ref{Skolem} for the case when the multiplicative group generated by $\a_1,\hdots,\a_m$ is torsion-free. (cf. \cite[Section 4.1]{BBF})
Similarly, we may assume that  the multiplicative group generated by $\e_1f,\hdots,\e_mf$ is torsion free.  Hence, it suffices to show the following for the proof of Theorem \ref{functionfields}.

\begin{theorem}\label{functionfields1}
 Let $S$ be a set of finitely many points of $C$ and let
  $\l_1,\hdots,\l_m\in \Ocal_S$  and   $f\in \Ocal_S^*$.
Fix an m-tuple of pairwise distinct integers $(r_1,\hdots,r_m)$   and let
$$
 B(n):=\l_1  f^{r_1n } +\cdots+\l_m   f^{r_mn}.
 $$
 Assume that for every positive integer $a$ there exists $k \in \ZZ$ such that
$v_\p(B(k))\ge \min\{1, v_\p(f^a -1)\}$ for all $\p\in C\setminus S$. Then there exists $n\in {\mathbb Z}$ such that $B(n)=0$.
 \end{theorem}

Without loss of generality, we assume that $r_1<r_2<\hdots<r_m$. Then we write
$$
B(n)=\l_m f^{r_1n} P(f^n),
$$
where
$$
P(T)=T^{r_m-r_1}+\frac{ \l_{m-1}}{ \l_{m}}T^{r_{m-1}-r_1}+\hdots+\frac{ \l_{1}}{ \l_{m}}.
$$
We assume that $P(f^k)\ne 0$ for all $k\in \ZZ$.  We also note that $P(f^k)=0$ is equivalent to  $B(k)=0$.
We will first find an integer $a$   and use the assumption of Theorem \ref{functionfields1} to obtain
 an integer $n$ such that $v_\p(P(f^n))\ge \min\{1, v_\p(f^a-1)\}$ for $\p\in C\setminus S$.
We then derive a contradiction with our assumption that $P(f^k)\ne 0$ for all $k\in \ZZ$. Consequently,
this shows that there exists an integer $k$ such that $P(f^k)=0$ and hence $B(k)=0$.

Similar to the number field case, we assume that $P(T)$ splits into linear factors in $\K$ by replacing $\K$ by a finite extension of $\K$.
We may also enlarge the size of $S$ such that all the zeros of $P(T)$ are $S$-units.

We split the polynomial $P(T)$ into two factors: $P(T)=P_{\rm ind}(T)P_{{\rm dep}}(T)$, such that each of the roots of $P_{\rm ind}(T)$
is multiplicatively independent of $f$, and those of
$P_{{\rm dep}}(T) $   are   multiplicatively  dependent with $f$.
 We take $q=2$ if all the roots of $P(T)$ are multiplicatively independent with $f$.  Otherwise, we may choose a smallest positive integer $q$
 such that $\b^q $ is a power of $f$ for every root $\b$  of  $P_{\rm dep }(T)$, i.e.
 $\b^q=f^{r_{\b}}$ with $r_{\b}\in\ZZ$.
We note that $q\ne 1$ since we have assumed that $P(f^k)\ne 0$ for all $k\in\ZZ$.
Now choose a prime integer $p$, not dividing $q$,  such that
  $r_{\b}-r_{\b' }$ is not divisible by $p$ for all  zeros $\b$ and $\b'$ of
$P_{{\rm dep}}(T)$ such that $r_{\b}\ne r_{\b' }$.
Recall that for a positive integer $k$,  we denote by $\Phi_{k}(T)$ the $k$-th cyclotomic polynomial, i.e.
$\Phi_{k}(T)=\prod_{\mu} (T-\mu)$, where $\mu$ runs over all primitive $k$-th roots of unity.
Let
$$a=p^{\ell }q,\quad\text{and }\quad g= \Phi_{p^{\ell}}(f)\Phi_{p^{\ell }q}(f),$$
where the positive integer $\ell$  will be specified later.
We need the following two lemmas.

\begin{lemma}\label{claimD}
  Let $n$ be an integer.  Then, either $N_S({\rm gcd}(P_{{\rm dep} }(f^n),\Phi_{p^{\ell}}(f)))=0$
 or $N_S({\rm gcd}(P_{{\rm dep}, }(f^n),\Phi_{p^{\ell}q}(f)))=0$.
 \end{lemma}

\begin{lemma}\label{claimI}
  $N_S(\gcd(P_{\rm ind}(f^n),g)) \le 3\sqrt[3]2\deg P\cdot \chi_S^{\frac13}\cdot [p^{\ell}qh(f)+h(P)]^{\frac23}$, where
$h(P):=h(\l_1,...,\l_m)$.
 \end{lemma}

We will now prove Theorem \ref{functionfields1}  by assuming Lemma \ref{claimD}  and Lemma \ref{claimI}.

\begin{proof}[Proof of Theorem \ref{functionfields1}]
Let $a=p^{\ell }q$   and  $g= \Phi_{p^{\ell}}(f)\Phi_{p^{\ell }q}(f)$  be as defined above.
First of all,  it follows from the assumption that there exists an integer $n$ such that $v_\p(P(f^n))\ge \min\{1, v_\p(f^a-1)\}$ for $\p\in C\setminus S$.
Since  $f^{a}-1$ is divisible by $ g$, it implies that
\begin{align*}
\overline{N}_S({\rm gcd}(P(f^n), g)&=\overline{N}_S( g)=\overline{N}_S(\Phi_{p^{\ell}}(f)\Phi_{p^{\ell }q}(f))\\
&=\sum_{\xi}\overline{N}_S(f- \xi)+\sum_{\eta}\overline{N}_S(f-\eta),
\end{align*}
where $\xi$ runs through all primitive $p^{\ell}$-th roots of unity and $\eta$ runs through all primitive $p^{\ell}q$-th roots of unity.
Together with Theorem \ref{SMTcff1}, we have
\begin{align}\label{lowerboundA}
\overline{N}_S({\rm gcd}(P(f^n), g))\ge (\varphi(p^{\ell})+ \varphi(p^{\ell}q)-2) h (f)-\chi_S.
\end{align}

Secondly, since $P(T)=P_{\rm ind}(T)P_{\rm dep}(T)$, we have
\begin{align}\label{uperboundA1}
N_S({\rm gcd}(P(f^n),g))\le N_S({\rm gcd}(P_{\rm ind}(f^n),g))+N_S({\rm gcd}(P_{\rm dep}(f^n),g)).
\end{align}
It follows from Lemma \ref{claimD}   that
\begin{align}\label{uperboundA2}
&N_S({\rm gcd}(P_{{\rm dep} }(f^n),g)) \cr
&=\max \{N_S({\rm gcd}(P_{{\rm dep}}(f^n),\Phi_{p^{\ell}}(f))),
N_S({\rm gcd}(P_{{\rm dep}}(f^n),\Phi_{p^{\ell}q}(f)))\}\cr
&\le \max \{N_S( \Phi_{p^{\ell}}(f)),
N_S( \Phi_{p^{\ell}q}(f))\}\cr
&\le \varphi(p^{\ell}q)h(f).
\end{align}
It then follows from (\ref{lowerboundA}),  (\ref{uperboundA2}) and  Lemma \ref{claimI} that  we have
\begin{align*}
(\varphi(p^{\ell}) -2) h (f)-\chi_S&\le 3\sqrt[3]2\deg P \chi_S^{\frac13}[p^{\ell}qh(f)+h(P)]^{\frac23}.
\end{align*}
Since $h(f)$ is a positive integer, this inequality is impossible   if $\ell$ is taken to be sufficiently large.
This contradiction implies that there must be an integer $n$ such that $P(f^n)=0$ and hence $B(n)=0$.
\end{proof}
It is now left to show Lemma \ref{claimD}  and Lemma \ref{claimI}.

\begin{proof}[Proof of Lemma  \ref{claimD}]
The assertion holds trivially, if $\deg P_{\rm dep}(T)=0$.  Therefore, we assume that $\deg P_{\rm dep}(T)\ge 1.$
Suppose that $N_S({\rm gcd}(P_{{\rm dep} }(f^n),\Phi_{p^{\ell}}(f)))\ne 0$
and $N_S({\rm gcd}(P_{{\rm dep}}(f^n),\Phi_{p^{\ell}q}(f)))\ne 0 $.
Then there exist two point $\p$ and $\q$ in $C\setminus S$ such that
$v_\p(P_{{\rm dep} }(f^n))$, $v_\p(\Phi_{p^{\ell}}(f))$, $v_\q(P_{{\rm dep}}(f^n))$ and $v_\q(\Phi_{p^{\ell}q}(f))$
are all positive.
Consequently,  $ f^n(\p)= \b(\p)$ for some $\b$ to be a root of
$P_{{\rm dep} }(T)$, and $f(\p)$ is a primitive $p^{\ell}$-th root of unity;
$ f^n(\q)=\widetilde{\b}(\q)$ for some $\widetilde{\b}$ to be a root of
$P_{{\rm dep} }(T)$, and $f(\q)$ is a primitive $p^{\ell}q$-th root of unity.
By the construction of $q$, we may write $\b^q=f^r$ and
$\widetilde{\b}^q=f^{r'}$.
Then $f^{nq}-\b^q=f^{nq}-f^r.$  Since $f$ is an $S$-unit, $f(\p)\ne 0$.
Then $ f^n(\p)= \b(\p)$ implies that
\begin{align}
 f^{nq-r}(\p)=1.
\end{align}
As $f(\p)$ is a primitive $p^{\ell}$-th root of unity,
$nq-r$ is divisible by $p^{\ell}$.
Similarly, the conditions on the point $\q$ imply that $nq-r'$ is divisible by $p^{\ell}q$.
Hence, $r-r'$ is divisible by $p^{\ell}$.  However, our choice of $p$ implies that
$r$ must equal $r'$.  Therefore, $nq-r $ is divisible by $p^{\ell}q$ and hence $r$ is divisible by $q$.
The relation $\b^q=f^r$ then implies that we may write $\b=\xi f^i$  where $i$ is an integer  and $\xi^q=1$.  Then $ f^n-\b=f^n-\xi f^i $.   Since $f$ is an $S$-unit, this implies that $v_\p( f^{n-i}-\xi)=v_\p( f^n-\b)>0$.  Therefore, $f^{n-i}(\p)=\xi$ and hence $f^{q(n-i)}(\p)=1$.  Then $p^{\ell}|q(n-i)$ since $f(\p)$ is a primitive $p^{\ell}$-th root of unity.  Since $q$ is not divisible by $p$, it follows that $p^{\ell}|(n-i)$. Then, $\xi=f^{n-i}(\p)=1$ which implies that $\b=f^i$ contradicting our assumption that $P(f^k)\ne 0$ for all $k\in\ZZ$.
\end{proof}

\begin{proof}[Proof of Lemma  \ref{claimI}]
Let $\b$ be a zero of $P_{\rm ind}(T)$.  We will first show that for any integer $n$,
\begin{align}\label{gcd1}
N_S(\gcd(f^n-\b, f^{p^{\ell}q}-1))\le3\sqrt[3]2 \chi_S^{\frac13}[p^{\ell}qh(f)+h(P)]^{\frac23}.
\end{align}
To show this assertion, we may assume that $0\le n\le p^{\ell}q$
since $n\equiv n'  {\rm mod}\, ( p^{\ell}q)$ implies that  $f^n\equiv f^{n'}\,  {\rm mod}\, ( f^{p^{\ell}q}-1)$ in the ring $\Ocal_S$.

 Since $\b$ is an $S$-unit by our assumption,
 $$
 N_S(\gcd(f^n-\b, f^{p^{\ell}q}-1))=N_S(\gcd(\b^{-1}f^n-1, f^{p^{\ell}q}-1)).
 $$
Theorem \ref{CZgcd} then implies that
\begin{align}\label{gcd2}
N_S(\gcd(f^n-\b, f^{p^{\ell}q}-1)) \le3\sqrt[3]2 \chi_S^{\frac13}[h(\b^{-1}f^n) h(f^{p^{\ell}q} ) ]^{\frac13}.
\end{align}
By basic properties of height functions, we have
\begin{align}\label{gcd3}
h(\b^{-1}f^n) h(f^{p^{\ell}q} )&\le p^{\ell}q h(f)(nh(f)+h(\b)) \cr
&\le p^{\ell}q h(f)(p^{\ell}qh(f)+h(P) )\qquad(\text{ by }  (\ref{Gausshtlinear}))\cr
&\le (p^{\ell}qh(f)+h(P) )^2
\end{align}
Then (\ref{gcd1}) follows from (\ref{gcd2}) and (\ref{gcd3}).
Since
\begin{align}
 N_S(\gcd(P_{\rm ind}(f^n),f^{p^{\ell}q}-1))  \le \sum_{\b} N_S(\gcd(f^n-\b, f^{p^{\ell}q}-1)),
\end{align}
where $\b$ runs through the zeros of $P_{\rm ind}(T)$, we may deduce from (\ref{gcd1})  that
\begin{align}
 N_S(\gcd(P_{\rm ind}(f^n),f^{p^{\ell}q}-1))  \le 3\sqrt[3]2\deg P \chi_S^{\frac13}[p^{\ell}qh(f)+h(P)]^{\frac23}.
\end{align}
The assertion of the Lemma then follows immediately since $f^{p^{\ell}q}-1$ is divisible by $g=\Phi_{p^{\ell}}(f)\Phi_{p^{\ell }q}(f)$.
\end{proof}

\subsection*{Acknowledgements}
The authors thank Felipe Voloch for bringing \cite{BBF} to their attention.
They also thank Chia-Liang Sun for helpful discussions.

\end{document}